\documentclass{birkjour}
\usepackage{amssymb,amscd,amsthm,amsmath}
\usepackage{amsfonts}
 \newtheorem{thm}{Theorem}[section]
 \newtheorem{cor}[thm]{Corollary}
 \newtheorem{lem}[thm]{Lemma}
 \newtheorem{pro}[thm]{Proposition}
 \theoremstyle{definition}
 \newtheorem{defn}[thm]{Definition}
 \theoremstyle{remark}
 \newtheorem{rem}[thm]{Remark}
 
 \numberwithin{equation}{section}
\usepackage{xcolor}

\begin{document}

%
%
%
%
%
%
%
%
%

\title[$\alpha$-BS dimension on  subsets]
 {$\alpha$-BS dimension on  subsets}

\author[Ding]{Zhumin Ding$^{1}$}

\address{%
School of Mathematical Sciences \\ Jiangxi Science and Technology Normal University \\ Nanchang 330038, P. R. China}

\email{dingzhumin1993@163.com}
\thanks{*Corresponding author. \\ The first author was supported by Jiangxi Provincial Natural Science Foundation (No. 20252BEJ730353) and Jiangxi Science and Technology Normal University Doctoral Research Initiation Fund (No. 2024BSQD34).  The second author was  supported by  the China Postdoctoral Science Foundation (No. 2024M763856) and  the Postdoctoral Fellowship Program of CPSF  (No. GZC20252040). The  third author   was  supported by the National Natural Science Foundation of China (No. 11971236) and Qinglan project of Jiangsu  Province.}

\author[Yang]{Rui Yang$^{2*}$}
\address{College of Mathematics and Statistics\\ Chongqing University\\ Chongqing, 401331, P.R.China}
\email{zkyangrui2015@163.com}
\author[Zhou]{Xiaoyao Zhou$^{3}$}
\address{School of Mathematical Sciences\\ Ministry of Education Key Laboratory of
NSLSCS\\ Nanjing Normal University\\ Nanjing 210023, Jiangsu, P.R.China}
\email{zhouxiaoyaodeyouxian@126.com}


\date{January 1, 2004}

\begin{abstract}
We aim to investigate the dimension theory of $\alpha$-pressure-like quantities.
By means of the Carath$\acute{\rm e}$odory-Pesin structure, we define $\alpha$-BS dimension and $\alpha$-Pesin topological pressure on subsets using $\alpha$-Bowen metric $$d_{n}^{\alpha}(x,y)=\max_{0\leq i\leq  n-1}e^{\alpha i}d(f^{i}x,f^{i}y),$$
where $\alpha \geq 0$. Specifically, we show that $\alpha$-BS dimension and $\alpha$-Pesin topological pressure are related by a Bowen's equation. Inspired by the classical Brin-Katok entropy, we introduce the notion of $\alpha$-local Brin-Katok entropy, and establish a variational principle for $\alpha$-BS dimension on compact subsets in terms of $\alpha$-local Brin-Katok entropy. Besides, for subshifts of finite type, we prove that $\alpha$-Bowen topological entropy is closely related to spectral radius and Hausdorff dimension.
\end{abstract}

\maketitle
\section{Introduction }

Let $(X,d,f)$ be a topological dynamical system (TDS for short), where $X$ is a compact metric space with a metric $d$, $f:X\rightarrow X$ is a continuous self-map.  The  space of all real-valued continuous  function is denoted by $C(X,\mathbb{R})$, endowed with the supremum norm $||\cdot||$. Denote by $M(X)$, $M(X,f)$ the sets of  Borel probability measures on $X$, and $f$-invariant  Borel probability measures on $X$, respectively.

 Given an abstract topological dynamical system, people are concerned about its dynamical behaviors. From the  quantitative point of view, one can introduce some  useful quantities to characterize the topological complexity of  dynamical systems. Kolmogorov \cite{kol58} and Sinai \cite{s59}  first realized  this idea  by introducing measure-theoretic entropy. Later, Adler, Konheim, and McAndrew \cite{A1}  introduced the concept of topological entropy. The variational principle \cite{din71,g0071,g0072} relates them by  stating that
 $$h_{top}(f,X)=\sup_{\mu \in M(X,f)}{h_{\mu}(f)},$$
 where $h_{top}(f,X)$ denotes the topological entropy of $X$,  and $h_{\mu}(f)$ is the measure-theoretic entropy of $\mu$. In 1973,  resembling the definition of Hausdorff dimension, Bowen \cite{b73} in his profound paper  further  extended the  concept of  topological entropy to arbitrary subsets of phase space, which we call Bowen topological entropy.  By treating the Bowen topological entropy    as  an analogue of dimension from the viewpoint of dimension theory,  the new ideas and techniques  from dimension theory can be injected to develop the entropy  theory of dynamical systems.  In 2012, Feng and Huang \cite{fh12}  established an analogous variational principle for Bowen topological entropy  in terms of  lower Brin-Katok local entropy. It states that for every  non-empty compact subset $K \subset X$,
 $$h_{top}^B(f,K)=\sup\{\underline{h}_{\mu}^{BK}(f): \mu \in \mathcal{M}(X), \mu(K)=1\},$$
 where $h_{top}^B(f,K)$ denotes the Bowen topological entropy of  $K$, and  $\underline{h}_{\mu}^{BK}(f)$ is the  lower Brin-Katok  local entropy of $\mu$. It turns out that  topological entropy, together with its variational principle, becomes a vital  tool in  the  investigating of  ergodic theory, multifractal analysis and  other fields of dynamical systems.

 Inspired by statistical mechanics, topological pressure, a generalization of topological entropy, was introduced by Ruelle \cite{rue73} for  certain dynamical systems, and later by  Walters for general dynamical systems \cite{w82}. As an extension of Bowen topological entropy,  Pesin and Piskel \cite{pp84} extended the notion of topological pressure to any subsets  using Carath$\acute{\rm e}$odory-Pesin structure. Readers are referred to the monographs \cite{w82,p97}  for the thermodynamic formalism theory of topological pressure and the related topics about the applications of topological pressure.

 In 1979, Bowen \cite{RB} demonstrated that the Hausdorff dimension of certain compact sets (quasi-circles) is the unique root of the equation defined by the topological pressure of a geometric potential function, which is also known as Bowen's equation. For continuous potential functions, Barreira and Schmeling \cite{bs00} introduced the concept of BS dimension, and  derived  a Bowen's equation for BS-dimension, proving it to be the unique root of the equation defined by the topological pressure of an additive potential function. It turns out that Bowen's equations provide a bridge between  thermodynamic formalism and dimension theory of dynamical systems, and  also serve as a powerful tool to estimate the Hausdorff dimension of the certain sets. See \cite{rue82,mm83,mu04,mu10,ms13} for some applications of Bowen's equations in this aspect. Besides, the authors in \cite{C1, DL} established the variational principles and Bowen's equations for BS dimension of $\mathbb{Z}$-actions and finitely generated free semi-group actions.

 Estimation entropy has its root in  control systems as the role of topological entropy played in  topological dynamical systems. The  estimation entropy, introduced  by  Liberzon and Mitra \cite{D2,D3}, measures  the smallest bit rate for exponential state estimation with a given exponent $\alpha\geq0$ for a continuous-time  control system on a compact subset of its state space. Kawan \cite{k13} generalized the estimation entropy to topological dynamical systems using  the $\alpha$-Bowen metric. It coincides with the classical topological entropy  \cite{w82} if  $\alpha=0$, but for the case of $\alpha> 0$, it depends on the metrics on the underlying space, and hence is not a purely topological invariant.  In \cite{ZC}, Zhong and Chen introduced the concept of Bowen estimation entropy for any subsets, and established an analogous Feng-Huang's variational principle for it.
 Recall that any two distinct points in the classical Bowen ball of the length $n$ with radius $\varepsilon$ are close within  a distance of  $\varepsilon$.  However, in the $\alpha$-Bowen ball of the length $n$ with radius $\varepsilon$, the distance between two distinct points is not just at most  $\frac{\varepsilon}{\mathrm{e}^{\alpha n}}$, and the distance between two distinct points can be sufficiently small for a long time. Therefore,  estimation entropy may increase the entropy (or disorder) of dynamical systems. This new phenomenon and the aforementioned work  mutually motivate us to ask whether there are  some possible extensions of $\alpha$-pressure-like quantities\footnote[1]{To show how dynamics of the systems are changed by varying the  parameter $\alpha>0$, we shall use  the terminology $\alpha$-pressure-like quantities instead of the estimation pressures.} to describe the topological complexity of dynamical systems. Moreover, the corresponding Bowen's equations and variational principles still hold for such $\alpha$-pressure-like quantities. In this paper, we contribute to the dimension theory of $\alpha$ pressure-like quantities. Using some functional analysis techniques and geometric  measure theory, we establish a Bowen's equation and a variational principle for $\alpha$-BS dimension, which provide the fundamental bridge between topological dynamics and dimension theory.

Next, we introduce the notion of $\alpha$-BS dimension on subsets.
Given $\varphi \in C(X,\mathbb{R})$ and  $x\in X$, we set $S_n\varphi(x):=\sum_{j=0}^{n-1}\varphi(f^jx)$. Given $\alpha\geq0$, $n\in\mathbb{N}$ and $x,y\in X$, the $\alpha$-Bowen metric on $X$  is defined by $$d_{n}^{\alpha}(x,y)=\max_{0\leq i\leq  n-1}e^{\alpha i}d(f^{i}x,f^{i}y).$$
The Bowen open ball of radius $\varepsilon$ centered at $x$ in the $\alpha$-Bowen metric is given by $$B_{n}^{\alpha}(x,\varepsilon)=\{y\in X:d_{n}^{\alpha}(x,y)<\varepsilon\}.$$

Analogous to the definition of BS dimension \cite{bs00}, by means of $\alpha$-Bowen balls we introduce BS dimension for any subsets using the theory of Carath$\acute{\rm e}$odory-Pesin structure \cite{p97}.

Let $K\subset X$,$\epsilon>0$, $\alpha\geq 0,s\geq 0$ and $n\in\mathbb{N}$, and $\varphi \in C(X,\mathbb{R})$  be a positive continuous function on $X$. Put
$$M^{\alpha}(K,s,\varepsilon,n,\varphi)=\inf\left\{\sum_{i}\exp(-s \cdot \sup_{y\in B_{n_i}^{\alpha}(x_i,\varepsilon)} S_{n_i}\varphi(y))\right\},$$
where the infimum is taken over all finite or countable families $\{B_{n_{i}}^{\alpha}(x_{i},\varepsilon)\}$ such that $x_i\in X,n_i\geq n$, and $\bigcup_{i}B_{n_{i}}^{\alpha}(x_{i},\varepsilon)\supseteq K$.

Here, the weight ${\rm exp}(- \sup_{y\in B_{n_i}^{\alpha}(x_i,\varepsilon)} S_{n_i}\varphi(y))$ can be regarded as a dynamical radius of the $\alpha$-Bowen ball $B_{n_{i}}^{\alpha}(x_{i},\varepsilon)$ as we have considered in  fractal geometry.

 Clearly, $M^{\alpha}(K,s,\varepsilon,n,\varphi)$ is non-decreasing as $n$ increases, and hence the following limit exists:
$$M^{\alpha}(K,s,\varepsilon,\varphi)=\lim_{n\rightarrow\infty}M^{\alpha}(K,s,\varepsilon,n,\varphi).$$
Denote by $M^{\alpha}(K,\varepsilon,\varphi)$ the critical value of $s$ such that $M^{\alpha}(K,s,\varepsilon,\varphi)$ jumps from $\infty$ to $0$, that is
$$M^{\alpha}(K,\varepsilon,\varphi)=\inf\{s:M^{\alpha}(K,s,\varepsilon,\varphi)=0\}=\sup\{s:M^{\alpha}(K,s,\varepsilon,\varphi)=\infty\}.$$

We define the \emph{$\alpha$-BS dimension}\footnote[2]{$\alpha$-BS dimension, as well as the subsequent $\alpha$-pressure(or entropy)-like quantities, depend on the metrics on $X$ although we do not clarify the metric $d$ in our notions, and remain unchanged if they take the set of  bi-Lipschitz metrics on $X$ (i.e., the set of all   compatible  metrics such that  the identity mapping $id:(X,d_1)\rightarrow (X,d_2)$ is a  bi-Lipschitz  mapping). Besides,  for $\alpha$-Bowen metric one can change the exponent ``$\mathrm{e}$" into any positive real number $\beta >1$. Following the similar procedure  and formulating the resulting $\alpha$-BS dimension ${\rm dim}_{BS}^{\alpha,\beta}(f,K,\varphi)$, we can conclude that ${\rm dim}_{BS}^{\alpha,\beta}(f,K,\varphi)={\rm dim}_{BS}^{\alpha\cdot \log\beta}(f,K,\varphi)$.} of $K$ with respect to (w.r.t.) $\varphi$ as
$${\rm dim}_{BS}^{\alpha}(f,K,\varphi)=\lim_{\varepsilon\rightarrow 0}M^{\alpha}(K,\varepsilon,\varphi).$$

We remark that this $\alpha$-BS dimension is a type of entropy, not a geometric (Hausdorff) dimension, and  has a close relation with the existing important concepts in dynamical systems and dimension theory. 

$(1)$  if $\varphi=1$,  the notion ${\rm dim}_{BS}^{\alpha}(f,K,1)$ recovers the $\alpha$-Bowen topological entropy $h_{top}^{\alpha,B}(f,K)$ of $f$ on $K$ \cite{ZC};

$(2)$   if $\alpha =0$, the notion ${\rm dim}_{BS}^{0}(f,K,\varphi)$ recovers the BS dimension introduced by Barreira and Schmeling \cite{bs00};

$(3)$  if $\varphi =1$ and $\alpha =0$,  the $\alpha$-BS dimension is reduced to  the Bowen topological entropy \cite{b73}.

The main results of this paper are as follows:

\begin{thm}\label{thm 1.1}
Let  $(X,d,f)$ be a TDS, $\alpha\geq0$ and $\varphi \in C(X,\mathbb{R})$ with $\varphi >0$. Suppose that $ h_{top}^{\alpha,B}(f,X)<\infty$. Then for any $K\subset X$,  ${\rm dim}_{BS}^{\alpha}(f,K,\varphi)$ is the unique root of the equation $$\Phi(t)=P_{B}^{\alpha}(f,K,-t\varphi)=0,$$
 and
 \begin{align*}
 {\rm dim}_{BS}^{\alpha}(f,K,\varphi)&=\inf\{t: P^{\alpha}_{B}(f,K,-t\varphi)\leq 0\}\\
 &=\sup\{t: P^{\alpha}_{B}(f,K,-t \varphi)\geq 0\},
 \end{align*}
 where  $P_{B}^{\alpha}(f,K,\varphi)$ denotes the $\alpha$-Pesin topological pressure of $\varphi$ on $K$.
\end{thm}

\begin{thm}\label{thm 1.2}
Let  $(X,d,f)$ be a TDS and $\varphi \in C(X,\mathbb{R})$ with $\varphi >0$. Let $K\subset X$ be a non-empty compact subset. Then  for every $\alpha\geq0$,
$${\rm dim}_{BS}^{\alpha}(f,K,\varphi)=\sup\{{P}_{\mu}^{\alpha}(f,\varphi):\mu\in\mathcal{M}(X),\mu(K)=1\},$$
where  ${P}_{\mu}^{\alpha}(f,\varphi)$ is  the $\alpha$-local  Brin-Katok entropy of $\mu$ w.r.t. $\varphi$.
\end{thm}

If $\alpha =0$, we remark that Theorem \ref{thm 1.1} is reduced to the Bowen's equation for BS dimension \cite[Proposition 6.4]{bs00}, and Theorem \ref{thm 1.2} is reduced to  the variational principle for BS-dimension \cite[Theorem 7.2]{C1}.


We calculate the $\alpha$-topological entropy of the subshifts of finite type and the Hausdorff dimension of the subsets of the phase space.

\begin{thm}\label{thm 1.4}
Let $(\Sigma_A^{+},d,\sigma)$ be a subshift of finite type generated by an incidence matrix $A$ of size $k\times k$. Then for every $\alpha \geq 0$,

$(1)$ For every $E \subset \Sigma_A^{+}$, $${\rm dim}_H(E,d) =\frac{h_{top}^{\alpha,B}(\sigma,E)}{1+\alpha},$$
where ${\rm dim}_H(E,d)$ is the Hausdorff dimension of $E$.

$(2)$ $h_{top}^{\alpha,B}(\sigma,\Sigma_A^{+})=h_{top}^{\alpha}(\sigma, \Sigma_A^{+}) = (1+\alpha) \log r(A),$
where $r(A)$ is the spectral radius of $A$.

Consequently, if $E=\Sigma_A^{+}$, then ${\rm dim}_H(\Sigma_A^{+},d)=\log r(A)$.
\end{thm}

The organization of the paper is as follows. In section \ref{sec:pre}, we derive the  fundamental properties of $\alpha$-BS dimension. In section \ref{sec:est},  we  prove Theorems \ref{thm 1.1}, \ref{thm 1.2}. In section \ref{sec 4}, we prove Theorem \ref{thm 1.4}. In section \ref{sec 5}, around the $\alpha$ entropy-like quantities we pose some open questions.


\section{$\alpha$-BS   dimension on subsets}\label{sec:pre}

In this section, we  derive the  fundamental properties of $\alpha$-BS dimension.

Actually, the aforementioned  definition for $\alpha$-BS dimension  can be alternatively defined  by changing the weight of $\alpha$-Bowen balls. Precisely, we change the $\sup_{y\in B_{n_i}^{\alpha}(x_i,\varepsilon)} S_{n_i}\varphi(y)$ into $ S_{n_i}\varphi(x_i)$ for each $\alpha$-Bowen ball $B_{n_i}^{\alpha}(x_i,\varepsilon)$, and then denote by the quantities $M^{\alpha,1}(K,s,\varepsilon,n,\varphi)$, $M^{\alpha,1}(K,s,\varepsilon,\varphi)$ and $M^{\alpha,1}(K,\varepsilon,\varphi)$, respectively. Also, we define $$\widetilde{\rm dim}_{BS}^{\alpha}(f,K,\varphi)=\lim_{\varepsilon \to 0}M^{\alpha,1}(K,\varepsilon,\varphi).$$
Thus $\widetilde{\rm dim}$ is the notion obtained by taking the value of $ S_{n}\varphi$ in the center $x_i$ of the respective $\alpha$-Bowen ball.
\begin{pro}
For any $\alpha \geq 0$, we have
$$\widetilde{\rm dim}_{BS}^{\alpha}(f,K,\varphi)={\rm dim}_{BS}^{\alpha}(f,K,\varphi). $$
\end{pro}

\begin{proof}
By comparing the definitions, we have $\widetilde{\rm dim}_{BS}^{\alpha}(f,K,\varphi)\geq {\rm dim}_{BS}^{\alpha}(f,K,\varphi)$.

Now we turn to verify  the reverse inequality $$\widetilde{\rm dim}_{BS}^{\alpha}(f,K,\varphi)\leq {\rm dim}_{BS}^{\alpha}(f,K,\varphi).$$
Let $\{B_{n_{i}}^{\alpha}(x_{i},\varepsilon)\}$ be a finite or countable family with  $x_i\in X,n_i\geq n$ and $\bigcup_{i}B_{n_{i}}^{\alpha}(x_{i},\varepsilon)\supseteq K$.
 Set $m=\min_{x\in X}\varphi(x)>0$ and $\gamma({\varepsilon})=\sup\{|\varphi(x)-\varphi(y)|:d(x,y)<\varepsilon\}$.
Since
$$\begin{aligned}
\sup_{y\in B_{n_i}^{\alpha}(x_i,\varepsilon)} S_{n_i}\varphi(y)
&\leq\sup_{y\in B_{n_i}^{\alpha}(x_i,\varepsilon)}\sum_{j=0}^{n_{i}-1}|\varphi(f^{j}y)-\varphi(f^{j}x_{i})|+ S_{n_i}\varphi(x_i)\\
&\leq n_{i}\gamma({\varepsilon})+ S_{n_i}\varphi(x_i).
\end{aligned}$$
Thus, using this inequality we obtain that
$$\begin{aligned}
\sum_{i}\exp\left(-s \cdot S_{n_i}\varphi(x_i)\right)
\leq\sum_{i}\exp\left(-s(1-\frac{\gamma({\varepsilon})}{m})\sup_{y\in B_{n_i}^{\alpha}(x_i,\varepsilon)} S_{n_i}\varphi(y)\right)
\end{aligned}$$
for all $s\geq 0$.
Then, by the definitions, it is clear that
$$(1-\frac{\gamma({\varepsilon})}{m})M^{\alpha,1}(K,\varepsilon,\varphi)\leq M^{\alpha}(K,\varepsilon,\varphi).$$
Letting $\varepsilon\rightarrow0$, we complete the proof.
\end{proof}

\begin{pro}\label{p22}
$\alpha$-BS dimension satisfies the following properties:

(1) ${\rm dim}_{BS}^{\alpha}(f,K,\varphi)\geq0$ for all $\alpha \geq 0$;

(2) if $K_{1}\subset K_{2}\subset X$, then ${\rm dim}_{BS}^{\alpha}(f,K_{1},\varphi)\leq {\rm dim}_{BS}^{\alpha}(f,K_{2},\varphi)$;

(3) if $K_{i}\subset X$ for $i=1,2,\ldots,$ then $${\rm dim}_{BS}^{\alpha}(f,\bigcup_{i=1}^{\infty} K_{i},\varphi)=\sup_{i\geq 1} {\rm dim}_{BS}^{\alpha}(f,K_{i},\varphi);$$

(4) let $\alpha \geq 0$ and $K\subset X$. If $h_{top}^{\alpha,B}(f,K)<\infty$, then for all $\varphi \geq 0$ $$\frac{h_{top}^{\alpha,B}(f,K)}{M}\leq {\rm dim}_{BS}^{\alpha}(f,K,\varphi)\leq \frac{h_{top}^{\alpha,B}(f,K)}{m},$$
where $m=\min_{x\in X}\varphi(x)$ and $M=\max_{x\in X}\varphi(x)$.
\end{pro}

\section{Proofs of main results}\label{sec:est}

In this section, we prove Theorems  $\ref{thm 1.1}$ and $\ref{thm 1.2}$.

\subsection{Bowen's equation for $\alpha$-BS  dimension}

In \cite{pp84}, Pesin defined the topological pressure on subsets. Using $\alpha$-Bowen balls, we continue to define the so-called $\alpha$-Pesin topological pressure on subsets, and link it to $\alpha$-BS dimension by a Bowen's equation.

Let $K\subset X$. For any $\alpha\geq 0$, $s\in \mathbb{R}$, $n\in\mathbb N$, $\varepsilon >0$, and the continuous function $h:X\rightarrow \mathbb{R}$, we define
$$\mathcal{M}_{B}^{\alpha}(K,s,\varepsilon,n,h)=\inf\left\{\sum_{i}\exp(-sn_{i}+\sup_{y\in B_{n_{i}}^{\alpha}(x_i,\varepsilon)} S_{n_i}h(y))\right\},$$
where the infimum is taken over all finite or countable families $\{B_{n_{i}}^{\alpha}(x_i,\varepsilon)\}$ with $x_i\in X$, $n_i\geq n$ such that $K\subset \bigcup_{i}B_{n_{i}}^{\alpha}(x_i,\varepsilon)$.

It is easy to see that $\mathcal{M}_{B}^{\alpha}(K,s,\varepsilon,n,h)$ is a finite outer measure on $X$ and increases as $n$ increases. Define
$$\mathcal{M}_{B}^{\alpha}(K,s,\varepsilon,h)=\lim_{n\rightarrow\infty}\mathcal{M}_{B}^{\alpha}(K,s,\varepsilon,n,h)$$
and
$$\mathcal{M}_{B}^{\alpha}(K,\varepsilon,h)=\inf\{s:\mathcal{M}_{B}^{\alpha}(K,s,\varepsilon,h)=0\}=\sup\{s:\mathcal{M}_{B}^{\alpha}(K,s,\varepsilon,h)=\infty\}.$$

\begin{defn}
The $\alpha$-Pesin topological pressure of $h$ on $K$ is defined by  $$P^{\alpha}_{B}(f,K,h)=\lim_{\varepsilon\rightarrow 0}\mathcal{M}_{B}^{\alpha}(K,\varepsilon,h).$$
\end{defn}

\begin{rem}
If $\alpha =0$, the above definition is reduced to the  Pesin  topological pressure \cite{ps07}. Additionally, if $h$ is a zero potential, it is reduced to  the $\alpha$-Bowen topological entropy.
\end{rem}

It is straightforward to show the $\alpha$-Pesin topological pressure possesses the following elementary properties:
\begin{pro} \label{p23}

(1) if $K_{1}\subset K_{2}\subset X$, then $P_{B}^{\alpha}(f,K_{1},h)\leq P_{B}^{\alpha}(f,K_{2},h)$;

(2) if $K_{i}\subset X$ for $i=1,2,\ldots,$ then $P_{B}^{\alpha}(f,\bigcup_{i=1}^{\infty} K_{i},h)=\sup_{i\geq 1} P_{B}^{\alpha}(f,K_{i},h)$;

(3)  $P^{\alpha}_{B}(f,K,h) >-\infty$ for all $h \in C(X,\mathbb{R})$. Furthermore, it is bounded by $\alpha$-Bowen topological entropy:
$$ h_{top}^{\alpha,B}(f,K)-||h||\leq P^{\alpha}_{B}(f,K,h)\leq  h_{top}^{\alpha,B}(f,K)+||h||;$$

(4) fix $\alpha \geq 0$ and $K\subset X$. If $h_{top}^{\alpha,B}(f,K)<\infty$, then
$$|P^{\alpha}_{B}(f,K,h_1)-P^{\alpha}_{B}(f,K,h_2)|\leq ||h_1-h_2||$$
for any $h_1,h_2\in C(X,\mathbb{R})$.
\end{pro}

We are ready to prove that the two different pressure-like quantities can be related by an analogue of Bowen's equation.

\begin{proof}[Proof of Theorem \ref{thm 1.1}]
We divide the proof into two steps:

\emph{Step 1.} We show that $$\Phi(t)=P_{B}^{\alpha}(f,K,-t\varphi)$$ is a strictly decreasing continuous function on $\mathbb{R}$.

 Notice that $|P^{\alpha}_{B}(f,K,h_1)-P^{\alpha}_{B}(f,K,h_2)|\leq ||h_1-h_2||$ by Proposition \ref{p23}. Then for  any $t_1,t_2  \in \mathbb{R}$, we have
$$||\Phi(t_1)-\Phi(t_2)||\leq ||\varphi||\cdot |t_1-t_2|.$$ This ensures the continuity of $\Phi$.

Let $m=\min_{x\in X}\varphi(x)>0$. For any $\varepsilon>0$, $t\in \mathbb{R}$ and $l>0$, we have
$$\begin{aligned}
&\mathcal{M}_{B}^{\alpha}(K,s,\varepsilon,n,-(t+l)\varphi)\\
=&\inf\left\{\sum_{i}\exp(-sn_{i}+\sup_{y\in B_{n_{i}}^{\alpha}(x_i,\varepsilon)}-(t+l)S_{n_i}\varphi(y))\right\}\\
\leq& \inf\left\{\sum_{i}\exp(-sn_{i}+\sup_{y\in B_{n_{i}}^{\alpha}(x_i,\varepsilon)}-tS_{n_i}\varphi(y)+\sup_{y\in B_{n_{i}}^{\alpha}(x_i,\varepsilon)}-lS_{n_i}\varphi(y))\right\}\\
\leq &\inf\left\{\sum_{i}\exp(-(s+ml)n_{i}+\sup_{y\in B_{n_{i}}^{\alpha}(x_i,\varepsilon)} -t S_{n_i}\varphi(y))\right\}\\
=&\mathcal{M}_{B}^{\alpha}(K,s+ml,\varepsilon,n,-t\varphi).
\end{aligned}$$
This implies that
$$\begin{aligned}
\mathcal{M}_{B}^{\alpha}(K,\varepsilon,-(t+l)\varphi)\leq \mathcal{M}_{B}^{\alpha}(K,\varepsilon,-t\varphi) -ml.
\end{aligned}$$
Letting $\varepsilon\rightarrow0$, we obtain 
$$P^{\alpha}_{B}(f,K,-(t+l)\varphi)<P^{\alpha}_{B}(f,K,-t\varphi).$$
Therefore, $\Phi(t)$ strictly decreases as $t$ increases.

\emph{Step 2.} We show that ${\rm dim}_{BS}^{\alpha}(f,K,\varphi)$ is the unique root of the equation $\Phi(t)=0$.

Observe from the definitions that for any $s \in \mathbb{R}$,
$$\mathcal{M}_{B}^{\alpha}(K,0,\varepsilon,n,-s\varphi)=M^{\alpha}(K,s,\varepsilon,n,\varphi).$$
Therefore, letting $n\rightarrow\infty$, we have
$$\mathcal{M}_{B}^{\alpha}(K,0,\varepsilon,-s\varphi)=M^{\alpha}(K,s,\varepsilon,\varphi).$$
Now let $s<{\rm dim}_{BS}^{\alpha}(f,K,\varphi)$. Then there exists $\varepsilon_0>0$ such that for all $0<\varepsilon<\varepsilon_0$, we have $s<M^{\alpha}(K,\varepsilon,\varphi)$. Hence,
$$\mathcal{M}_{B}^{\alpha}(K,0,\varepsilon,-s\varphi)=M^{\alpha}(K,s,\varepsilon,\varphi)=\infty.$$
This shows that $P^{\alpha}_{B}(f,K,-s\varphi)\geq0$.  Letting  $s \to  {\rm dim}_{BS}^{\alpha}(f,K,\varphi)$ yields that
$$P^{\alpha}_{B}(f,K,-{\rm dim}_{BS}^{\alpha}(f,K,\varphi) \cdot\varphi)\geq0.$$
The similar arguments give the reverse inequality $P^{\alpha}_{B}(f,K,-{\rm dim}_{BS}^{\alpha}(f,K,\varphi) \cdot\varphi)\leq0.$ Since $\Phi(t)$ is strictly decreasing in $t$, so  ${\rm dim}_{BS}^{\alpha}(f,K,\varphi)$ is the unique root of the equation $\Phi(t)=0$.
\end{proof}

\begin{rem}
There is no geometric condition assumed on $f$ such as expanding or hyperbolic in Theorem \ref{thm 1.1}, thus the result is not geometric (Hausdorff or box dimension). By the above the process of Theorem 1.1, we know that the result follows because $\Phi(t)$ is continuous and strictly decreasing.
\end{rem}

\subsection{Variational principle for $\alpha$-BS dimension}

In this subsection, we establish the variational principle between $\alpha$-BS dimension and $\alpha$-local Brin-Katok entropy.

\subsubsection{Billingsley Type Theorem for $\alpha$-BS dimension}

Inspired by the classical Brin-Katok entropy \cite{bk83}, we introduce the notion of $\alpha$-local Brin-Katok entropy.

\begin{defn}
Let $\mu\in\mathcal{M}(X)$, $\alpha\geq0$ and $\varphi:X\rightarrow\mathbb{R}$ be a positive continuous function. The $\alpha$-local  Brin-Katok entropy of $\mu$ w.r.t. $\varphi$ is defined by
$${P}_{\mu}^{\alpha}(f,\varphi)=\int{{P}_{\mu}^{\alpha}(f,x,\varphi)d\mu(x)},$$
where
$${P}_{\mu}^{\alpha}(f,x,\varphi)=\lim_{\varepsilon\rightarrow0}\liminf_{n\rightarrow\infty}
-\frac{\log\mu(B_{n}^{\alpha}(x,\varepsilon))}{ S_{n}\varphi(x)}.$$
\end{defn}
The following Billingsley-type theorem illustrates  how the $\alpha$-BS dimension of the subsets of $X$ is determined by the $\alpha$-local  Brin-Katok entropy of Borel probability measures on $X$.

We invoke a 3r-covering lemma for our forthcoming proof.

\begin{lem}\cite{ZC}\label{l32}
Let $\alpha\geq0$, $\varepsilon>0$, $\mathcal{B}^{\alpha}(\varepsilon)=\{B_{n}^{\alpha}(x,\varepsilon):x\in X, n=1,2,\ldots\}$. For any family $\mathcal{F}\subset\mathcal{B}^{\alpha}(\varepsilon)$, there exists a subfamily $\mathcal{G}\subset\mathcal{F}$ consisting of disjoint balls such that
$$\bigcup_{B\in\mathcal{F}}B\subset\bigcup_{B_{n}^{\alpha}(x,\varepsilon)\in\mathcal{G}} B_{n}^{\alpha}(x,3\varepsilon).$$
\end{lem}

\begin{thm}\label{t33}
Let $\mu\in\mathcal{M}(X)$, $K\subset X$ be a Borel set, $\alpha\geq0$ and $\varphi:X\rightarrow\mathbb{R}$ be a positive continuous function. For each $s\in (0,\infty)$, the following statements hold:

$(1)$ if ${P}_{\mu}^{\alpha}(f,x,\varphi)\leq s$ for all $x\in K$, then ${\rm dim}_{BS}^{\alpha}(f,K,\varphi)\leq s$,

$(2)$ if ${P}_{\mu}^{\alpha}(f,x,\varphi)\geq s$ for all $x\in K$ and $\mu(K)>0$, then ${\rm dim}_{BS}^{\alpha}(f,K,\varphi)\geq s$.
\end{thm}

\begin{proof}
(1).  Given $\beta>0$, let $$K_{m}=\left\{x\in K:\liminf_{n\rightarrow\infty}\frac{-\log\mu(B_{n}^{\alpha}(x,\varepsilon))}{ S_{n}\varphi(x)}<s+\beta,\ \text{for any}\ \varepsilon\in(0,\frac{1}{m})\right\}.$$
Since ${P}_{\mu}^{\alpha}(f,x,\varphi)\leq s$ for all $x\in K$, then $K=\bigcup_{m=1}^{\infty}K_{m}$. Fix $m\geq1$ and $\varepsilon\in(0,\frac{1}{3m})$. For any $x\in K_m$, there exists a strictly increasing sequence $\{n_{j}(x)\}_{j=1}^{\infty}$ such that
$$\mu(B_{n_{j}(x)}^{\alpha}(x,\varepsilon))\geq\exp{\left(-(s+\beta)S_{n_j}\varphi(x)\right)}\ \text{for all}\ j\geq1.$$
Put $\mathcal{F}_{N}=\{B_{n_{j}(x)}^{\alpha}(x,\varepsilon):x\in K_{m}, n_{j}(x)\geq N\}$. It is obvious that $K_{m}\subset \bigcup_{B\in\mathcal{F}_{N}}B.$ Using Lemma \ref{l32}, there exists a subfamily $\mathcal{F}_{N}^{'}=\{B_{n_{i}}^{\alpha}(x_{i},\varepsilon)\}_{i\in I}\subset\mathcal{F}_{N}$ consisting of pairwise disjoint balls such that for all $i\in I$,
$$K_{m}\subset\bigcup_{i\in {I}}B_{n_{i}}^{\alpha}(x_{i},3\varepsilon).$$

Notice that
$$\mu(B_{n_{i}}^{\alpha}(x_{i},\varepsilon))\geq\exp{\left(-(s+\beta) S_{n_i}\varphi(x_i)\right)}>0\ \text{for every}\ i\in {I}.$$
So  the index set $\mathcal{I}$ is at most countable. It follows that
$$M^{\alpha}(K_{m},s+\beta,3\varepsilon,N,\varphi)\leq\sum_{i\in {I}}\exp{\left(-(s+\beta) S_{n_i}\varphi(x_i)\right)}
\leq\sum_{i\in {I}}\mu(B_{n_{i}}^{\alpha}(x_{i},\varepsilon))\leq1.$$
This implies that ${\rm dim}_{BS}^{\alpha}(f,K_m,\varphi)\leq s+\beta$ for any $m\geq1$. By Proposition \ref{p22}, we get $${\rm dim}_{BS}^{\alpha}(f,K,\varphi)=\sup_{m\geq1}{\rm dim}_{BS}^{\alpha}(f,K_m,\varphi)\leq s+\beta.$$
Therefore, we have ${\rm dim}_{BS}^{\alpha}(f,K,\varphi)\leq s$ by the arbitrariness of $\beta$.

(2) Fix $\beta>0$. For any $m\geq1$, let
$$K_{m}=\left\{x\in K:\liminf_{n\rightarrow\infty}\frac{-\log\mu(B_{n}^{\alpha}(x,\varepsilon))}{ S_{n}\varphi(x)}>s-\beta,\ \text{for any}\ \varepsilon\in(0,\frac{1}{m}]\right\}.$$
Since $\frac{-\log\mu(B_{n}^{\alpha}(x,\varepsilon))}{ S_{n}\varphi(x)}$ is non-decreasing  as $\varepsilon$ decreases, then
$$K_{m}=\left\{x\in K:\liminf_{n\rightarrow\infty}\frac{-\log\mu(B_{n}^{\alpha}(x,\frac{1}{m}))}{ S_{n}\varphi(x)}>s-\beta \right\}.$$
Thus, $K_{m}\subset K_{m+1}$ and $\bigcup_{m=1}^{\infty}K_m=K$. For  sufficiently large $M\geq1$, we have $\mu(K_{M})>\frac{1}{2}\mu(K)>0$. For each $N\geq1$, set
$$\begin{aligned}
K_{M,N}
&:=\left\{x\in K_M:\frac{-\log\mu(B_{n}^{\alpha}(x,\varepsilon))}{ S_{n}\varphi(x)}>s-\beta\ \text{for any}\ n\geq N, \varepsilon\in(0,\frac{1}{M}]\right\}\\
&=\left\{x\in K_M:\frac{-\log\mu(B_{n}^{\alpha}(x,\frac{1}{M}))}{ S_{n}\varphi(x)}>s-\beta\ \text{for any}\ n\geq N\right\}.
\end{aligned}$$
Clearly, $K_{M,N}\subset K_{M,N+1}$ and $\bigcup_{N=1}^{\infty}K_{M,N}=K_{M}$. Similarly, we can take an $N^{*}\geq1$ such that $\mu(K_{M,N^{*}})>\frac{1}{2}\mu(K_{M})>0$. Then for any $x\in K_{M,N^{*}}$,  we have
$$\mu(B_{n}^{\alpha}(x,\varepsilon))\leq\exp\left({-(s-\beta)} S_{n}\varphi(x)\right)$$
 for all $n\geq N^{*}$ and $0<\varepsilon\leq \frac{1}{M}$.

Fix $N\geq N^{*}$. Without loss of generality, assume that $\{B_{n_{i}}^{\alpha}(y_{i},\frac{\varepsilon}{2})\}$ is a finite or countable family such that  $y_i\in X,n_i\geq N$, $\bigcup_{i}B_{n_{i}}^{\alpha}(y_{i},\frac{\varepsilon}{2})\supseteq K_{M,N^{*}}$ and  $$K_{M,N^{*}}\bigcap B_{n_{i}}^{\alpha}(y_{i},\frac{\varepsilon}{2})\neq\emptyset,\ \text{for all}\ i\geq1,$$
where $ 0<\varepsilon<\frac{1}{M}.$
For any $i\geq1$, take $x_{i}\in K_{M,N^{*}}\bigcap B_{n_{i}}^{\alpha}(y_{i},\frac{\varepsilon}{2})$ such that
$$B_{n_{i}}^{\alpha}(y_{i},\frac{\varepsilon}{2})\subset B_{n_{i}}^{\alpha}(x_{i},\varepsilon).$$
Thus
$$0<\mu(K_{M,N^{*}})\leq\sum_{i\geq1}\mu(B_{n_{i}}^{\alpha}(x_{i},\varepsilon))\leq
\sum_{i\geq1}\exp\left({-(s-\beta)}S_{n_i}\varphi(x_i)\right).$$
Then
$$M^{\alpha}(K_{M,N^{*}},s-\beta,\frac{\varepsilon}{2},N,\varphi)\geq\mu(K_{M,N^{*}})>0,$$
which implies that
$${\rm dim}_{BS}^{\alpha}(f,K,\varphi)\geq {\rm dim}_{BS}^{\alpha}(f,K_{M,N^{*}},\varphi)\geq s-\beta.$$
The arbitrariness of $\beta$ implies that ${\rm dim}_{BS}^{\alpha}(f,K,\varphi)\geq s$.
\end{proof}

In the following, we connect the $\alpha$-BS dimension to the $\alpha$-local Brin-Katok entropy of $\mu$ by the variational principle on  compact subsets.   We first present the notion of weighted $\alpha$-BS dimension, which plays a crucial role in the proof of Theorem \ref{thm 1.2}.

\subsubsection{Weighted $\alpha$-BS dimension}

Let $\alpha\geq0$ and $\varphi:X\rightarrow\mathbb{R}$ be a positive continuous function. For any $n\in\mathbb{N}$, $s\in\mathbb{R}$, $\varepsilon>0$ and a  bounded function $\xi:X\rightarrow\mathbb{R}$, we define
$$W^{\alpha}(\xi,s,\varepsilon,n,\varphi)=\inf\left\{\sum_{i}c_i\exp\left(-s \cdot  S_{n_i}\varphi(x_i)\right)\right\},$$
where the infimum is taken over all finite or countable collections $\{(B_{n_{i}}^{\alpha}(x_i,\varepsilon),c_{i})\}$ such that $x_{i}\in X$, $n_i\geq n$, $0<c_i<\infty$ and $$\sum_{i}c_{i}\chi_{B_{i}}\geq \xi,$$ where $B_{i}:=B_{n_{i}}^{\alpha}(x_i,\varepsilon)$ and $\chi_{B_{i}}$ denotes the characteristic function of $B_{i}$.

For $K\subset X$ and $\xi=\chi_K$, we let $ W^{\alpha}(K,s,\varepsilon,n,\varphi):=W^{\alpha}(\chi_K,s,\varepsilon,n,\varphi)$ and define
$$W^{\alpha}(K,s,\varepsilon,\varphi)=\lim_{n\rightarrow\infty}W^{\alpha}(K,s,\varepsilon,n,\varphi).$$
There exists a critical value of the parameter $s$ such that $W^{\alpha}(K,s,\varepsilon,\varphi)$ jumps from $\infty$ to $0$, which we  denote by $W^{\alpha}(K,\varepsilon,\varphi)$, that is
$$W^{\alpha}(K,\varepsilon,\varphi)=\inf\{s:W^{\alpha}(K,s,\varepsilon,\varphi)=0\}=\sup \{s:W^{\alpha}(K,s,\varepsilon,\varphi)=\infty\}.$$

\begin{defn}
The weighted $\alpha$-BS   dimension of $K$ is defined by $${\rm dim}_{BS}^{\alpha,W}(f,K,\varphi)=\lim_{\varepsilon\rightarrow0}W^{\alpha}(K,\varepsilon,\varphi).$$
\end{defn}

We show that the weighted $\alpha$-BS dimension coincides with the  $\alpha$-BS dimension by the classical 5r-covering lemma. \begin{lem}\label{l35}\cite[Theorem 2.1]{PM}
Let $(X,d)$ be a compact metric space and $\mathcal{B}=\{B(x_i,r_i)\}_{i\in\mathcal {I}}$ be a family of closed (or open) balls in $X$. Then there exists a finite or countable subfamily $\mathcal{B^{'}}=\{B(x_i,r_i)\}_{i\in\mathcal {I^{'}}}$ of pairwise disjoint balls in $\mathcal{B}$ such that
$$\bigcup_{B\in\mathcal{B}}B\subset \bigcup_{i\in\mathcal{I^{'}}}B(x_i,5r_i).$$
\end{lem}

\begin{thm}\label{t34}
For $K\subset X$, $\alpha\geq0$ and $\varphi  \in C(X,\mathbb{R})$ with $\varphi >0$, we have $${\rm dim}_{BS}^{\alpha,W}(f,K,\varphi)={\rm dim}_{BS}^{\alpha}(f,K,\varphi).$$
\end{thm}

\begin{proof}
The inequality ${\rm dim}_{BS}^{\alpha,W}(f,K,\varphi)\leq{\rm dim}_{BS}^{\alpha}(f,K,\varphi)$  is obvious by definitions.

Next, for obtaining the reverse inequality we prove that $$ M^{\alpha}(K,s+\delta,6\varepsilon,n,\varphi)\leq W^{\alpha}(K,s,\varepsilon,n,\varphi)$$ for sufficiently large $n$.

Assume that $n>2$ such that $m^{2}\exp(-\gamma m\delta)\leq1$ for  all $m\geq n$, where $\gamma=\min_{x\in X}\varphi(x)>0$. Let $\{(B_{n_{i}}^{\alpha}(x_{i},\varepsilon),c_i)\}_{i\in {I}}$ be a family such that  $x_i\in X$, $0<c_i<\infty$, $n_i\geq n$ and $\sum_{i\in {I}}c_{i}\chi_{B_{i}}\geq\chi_K$, where $B_{i}=B_{n_{i}}^{\alpha}(x_{i},\varepsilon)$. We need to show
\begin{align}\label{equ 3.1}
M^{\alpha}(K,s+\delta,6\varepsilon,n,\varphi)\leq\sum_{i\in {I}}c_{i}\exp\left(-s \cdot  S_{n_i}\varphi(x_i)\right).
\end{align}

Put $I_m=\{i\in {I}:n_{i}=m\}$ and ${I}_{m,k}=\{i\in {I}_{m}:i\leq k\}$ for any $m\geq n$ and $k\in\mathbb N$.  We write $B_{i}=B_{n_{i}}^{\alpha}(x_{i},\varepsilon)$, $5B_{i}=B_{n_{i}}^{\alpha}(x_{i},5\varepsilon)$ for $i\in {I}$. Without loss of generality, we  may assume that $B_i\neq B_j$ for $i\neq j$. For $t>0$, set
$$K_{m,t}=\left\{x\in K:\sum_{i\in {I}_{m}}c_{i}\chi_{B_{i}}(x)>t\right\},$$
$$K_{m,k,t}=\left\{x\in K:\sum_{i\in {I}_{m,k}}c_{i}\chi_{B_{i}}(x)>t\right\}.$$
We shall prove (\ref{equ 3.1}) by  the following three steps.

Step 1. For each $m\geq n$, $k\in\mathbb N$, and $t>0$, there exists a finite set ${J}_{m,k,t}\subset{I}_{m,k}$ such that the balls $B_i$, indexed by $i\in{J}_{m,k,t}$,  are pairwise disjoint, $K_{m,k,t}\subset\bigcup_{i\in{J}_{m,k,t}}5B_{i}$ and
$$\sum_{i\in {J}_{m,k,t}}\exp\left(-s \cdot  S_{m}\varphi(x_i)\right)
\leq\frac{1}{t}\sum_{i\in {I}_{m,k}}c_{i}\exp\left(-s \cdot  S_{m}\varphi(x_i)\right).$$
We adopt the proof from  Federer \cite{HF} and   Mattila \cite[Lemma 8.16]{PM} to   complete the proof. Since ${I}_{m,k}$ is finite, by approximating the $c_i$'s from above, we  may assume that each $c_i$ is a positive rational number, and  even is a positive integer  by  multiplying with a common denominator of these rational numbers. Let $h$ be the least integer with $h\geq t$. Denote $\mathcal{B}=\{B_i:i\in {I}_{m,k}\}$ and define $u:\mathcal{B}\rightarrow\mathbb{Z}$ by $u(B_i)=c_i$. Since $B_i\neq B_j$ for $i\neq j$, $u$ is well-defined.

By induction, we define  the integer-valued functions $v_0,v_1,\ldots,v_h$ on $\mathcal{B}$ and  the subfamilies $\mathcal{B}_1,\mathcal{B}_2,\ldots,\mathcal{B}_h$ of $\mathcal{B}$ starting with $v_0=u$. By Lemma \ref{l35}, there is a pairwise disjoint subfamily $\mathcal{B}_1$ of $\mathcal{B}$ such that $\bigcup_{B\in\mathcal{B}}B\subset\bigcup_{B\in\mathcal{B}_1}5B$ and $K_{m,k,t}\subset\bigcup_{B\in\mathcal{B}_1}5B$. Then, by repeatedly using Lemma \ref{l35}, for $j=1,2,\ldots, h$, we inductively define   disjoint subfamilies $\mathcal{B}_{j}$ of $\mathcal B$ such that
$$\mathcal{B}_{j}\subset\{B\in\mathcal{B}: v_{j-1}(B)\geq 1\},\ \ K_{m,k,t}\subset \bigcup_{B\in\mathcal{B}_j}5B$$
and the function $v_j$ such that
$$v_j(B)=\left\{
\begin{aligned}
v_{j-1}(B)-1,\ \text{for}\ B\in\mathcal{B}_{j}, \\
v_{j-1}(B),\ \text{for}\ B\in\mathcal{B}\backslash\mathcal{B}_{j}.
\end{aligned}
\right.$$
This is possible since for $j<h$,
$$K_{m,k,t}\subset\left\{x:\sum_{B\in\mathcal{B}:B\ni x}v_{j}(B)\geq h-j\right\},$$
whence every $x\in K_{m,k,t}$ belongs to some ball $B\in\mathcal{B}$ with $v_{j}(B)\geq 1$. Thus
$$\begin{aligned}
\sum_{j=1}^{h}\sum_{B\in\mathcal{B}_{j}}\exp\left(-s \cdot  S_{m}\varphi(x_B)\right)
&=\sum_{j=1}^{h}\sum_{B\in\mathcal{B}_{j}}(v_{j-1}(B)-v_{j}(B))\exp\left(-s \cdot  S_{m}\varphi(x_B)\right)\\
&\leq \sum_{B\in\mathcal{B}}\sum_{j=1}^{h}(v_{j-1}(B)-v_{j}(B))\exp\left(-s \cdot  S_{m}\varphi(x_B)\right)\\
&\leq\sum_{B\in\mathcal{B}}u(B)\exp\left(-s \cdot  S_{m}\varphi(x_B)\right)\\
&=\sum_{j\in{I}_{m,k}}c_{j}\exp\left(-s \cdot  S_{m}\varphi(x_j)\right).
\end{aligned}$$
Choose $j_0\in\{1,2,\ldots,h\}$ so that $\sum_{B\in\mathcal{B}_{j_0}}\exp\left(-s \cdot  S_{m}\varphi(x_B)\right)$ is the smallest term. Then
$$\begin{aligned}
\sum_{B\in\mathcal{B}_{j_0}}\exp\left(-s\cdot S_m\varphi(x_B)\right)
\leq\frac{1}{h}\sum_{j\in{I}_{m,k}}c_{j}\exp\left(-s \cdot  S_{m}\varphi(x_j)\right)\\
\leq\frac{1}{t}\sum_{j\in{I}_{m,k}}c_{j}\exp\left(-s \cdot  S_{m}\varphi(x_j)\right).
\end{aligned}$$
Thus ${J}_{m,k,t}=\{i\in {I}:B_{i}\in\mathcal{B}_{j_0}\}$ is as desired.

Step 2. For each $m\geq n$ and $t>0$, we have
$$M^{\alpha}(K_{m,t},s+\delta,6\varepsilon,n,\varphi)\leq
\frac{1}{m^{2}t}\sum_{i\in {I}_{m}}c_i\exp\left(-s \cdot  S_{m}\varphi(x_i)\right).$$

We may  assume that $K_{m,t}\neq\emptyset$; otherwise there is nothing left to prove. Since $K_{m,k,t}$ is monotonically increasing to $K_{m,t}$, then $K_{m,k,t}\neq\emptyset$ for sufficiently large $k$. Fix such $k$. By Step 1,  the  set ${J}_{m,k,t}\neq\emptyset$ when $k$ is large enough. Let $Z_{m,k,t}=\{x_i:i\in{J}_{m,k,t}\}$. Notice  that the space of all non-empty compact subset of $X$ is compact with respect to Hausdorff distance (cf. Federer \cite{HF} 2.10.21). Hence, there is a sub-sequence $\{k_j\}\subset\mathbb{N}$ and a non-empty compact subset $Z_{m,t}\subset X$ such that $Z_{m,k_{j},t}$ converges to $Z_{m,t}$ in the Hausdorff distance $d_H$ as $j\rightarrow\infty$. As any two points in $Z_{m,k,t}$ have a distance (with respect to $d_{m}^{\alpha}$) no less than $\varepsilon$, so do the points in $Z_{m,t}$. Hence, $Z_{m,t}$ is a finite set. Moreover, $\sharp(Z_{m,k_{j},t})=\sharp(Z_{m,t})$ when $j$ is large enough.  Choose $\delta <\frac{1}{3}\min_{x\not= y,x,y\in Z_{m,t}}d(x,y)$ such that for any $x,y\in X$ with $d(x,y)<\delta$, one has $d_m^{\alpha}(x,y)<\frac{\varepsilon}{2}$. Notice that for sufficiently large $j$,
$$d_H(Z_{m,k_j,t},Z_{m,t})<\delta,$$
and
$$K_{m,k_j,t}\subset\bigcup_{i\in {J}_{m,k_{j},t}}5B_{i}=\bigcup_{x\in Z_{m,k_{j},t}}B_{m}^{\alpha}(x,5\varepsilon) \subset\bigcup_{x\in Z_{m,t}}B_{m}^{\alpha}(x,5.5\varepsilon).$$
One has $K_{m,t}\subset\bigcup_{x\in Z_{m,k_j,t}}B_{m}^{\alpha}(x,6\varepsilon)$.
Therefore, by Step 1  we  obtain that
$$\begin{aligned}
M^{\alpha}(K_{m,t},s+\delta,6\varepsilon,n,\varphi)
&\leq\sum_{x\in Z_{m,k_j,t}}\exp\left(-(s+\delta)S_m \varphi (x)\right)\\
&\leq \frac{1}{\exp(\gamma m\delta)t}\sum_{i\in {I}_m}c_i\exp\left(-s \cdot  S_{m}\varphi(x_i)\right)\\
&\leq\frac{1}{m^{2}t}\sum_{i\in {I}_m}c_i\exp\left(-s \cdot  S_{m}\varphi(x_i)\right).
\end{aligned}$$
Step 3. For any $t\in (0,1)$, we have
$$M^{\alpha}(K,s+\delta,6\varepsilon,n,\varphi)
\leq\frac{1}{t}\sum_{i\in {I}}c_{i}\exp\left(-s \cdot  S_{n_i}\varphi(x_i)\right).$$
 Fix $t\in(0,1)$. Note that $\sum_{m=n}^{\infty}m^{-2}<1$. Then $$K\subset\bigcup_{m=n}^{\infty}K_{m,m^{-2}t}.$$ Hence, by Step 2 we get
$$\begin{aligned}
M^{\alpha}(K,s+\delta,6\varepsilon,n,\varphi)
\leq&\sum_{m=n}^{\infty}M^{\alpha}(K_{m,m^{-2}t},s+\delta,6\varepsilon,n,\varphi)\\
\leq&\sum_{m=n}^{\infty}\frac{1}{t}\sum_{i\in  {I}_{m}}c_{i}\exp\left(-s \cdot  S_{m}\varphi(x_i)\right)\\
=&\frac{1}{t}\sum_{i\in  I}c_{i}\exp\left(-s \cdot  S_{n_i}\varphi(x_i)\right),\\
\end{aligned}$$
which implies that $M^{\alpha}(K,s+\delta,6\varepsilon,n,\varphi)\leq\sum_{i\in {I}}c_{i}\exp\left(-s \cdot  S_{n_i}\varphi(x_i)\right)$ by letting $t \to 1$.
This completes the proof.
\end{proof}

\subsubsection{Frostman's lemma of  weighted $\alpha$-BS    dimension}

The  following  dynamical version of Frostman's lemma  allows us  to  estimate the   measure  of $\alpha$-Bowen balls.
\begin{lem}\label{l36}
Let $\alpha\geq0$, $s>0$, $n\in\mathbb{N}$, $\varepsilon>0$ and $\varphi:X\rightarrow\mathbb{R}$ be a positive continuous function.
Suppose that $K$ is a non-empty compact subset of $X$.  If $c:=W^{\alpha}(K,s,\varepsilon,n,\varphi)>0$, then there exists a Borel probability measure $\mu$ on $X$ with $\mu(K)=1$   satisfying for any $x\in X$ and  $m\geq n$, $$\mu(B_{m}^{\alpha}(x,\varepsilon))\leq\frac{1}{c}\exp\left(-s \cdot S_mf(x)\right).$$
\end{lem}

\begin{proof}
Notice that $0<c<\infty$. We define a function $p:C(X, \mathbb{R})\rightarrow \mathbb{R}$ by $$p(g)=\frac{1}{c}W^{\alpha}(\chi_{K}\cdot g,s,\varepsilon,n,\varphi).$$
Then it is easy to see that the function $p$ satisfies the following properties:

(1) $p(f+g)\leq p(f)+p(g)\ \text{for any}\ f,g\in C(X, \mathbb{R})$;

(2) $p(tg)=tp(g)$ for any $t\geq 0$ and $g\in C(X, \mathbb{R})$;

(3) $p(\mathbf{1})=1$, $\mathbf{1}$ is the constant function $\mathbf{1}(x)=1$,  $0\leq p(g)\leq \|g\|$ for any $g\in C(X, \mathbb{R})$ and $p(f)=0$ for all $f\in C(X, \mathbb{R})$ with $f\leq 0$.

Applying the Hahn-Banach theorem, we can extend the linear function $t\rightarrow t p(\mathbf{1})$, $t\in\mathbb R$, from the subspace of the constant functions to a linear functional $\Psi:C(X, \mathbb{R})\rightarrow \mathbb R$ satisfying
$$\Psi(\mathbf{1})=p(\mathbf{1})=1\ \text{and}\ -p(-g)\leq\Psi(g)\leq p(g)\ \text{for any}\ g\in C(X, \mathbb{R}).$$

If $g\in C(X, \mathbb{R})$ with $g\geq 0$, then $p(-g)=0$, and hence $\Psi(g)\geq 0$. By the  Riesz Representation Theorem,  there exists $\mu\in M(X)$ such that $\Psi(g)=\int_{X}gd\mu$ for $g\in C(X, \mathbb{R})$.

Now we show  $\mu(K)=1$. To see this, for any compact set $E\subset X\backslash K$, by Urysohn's lemma there exists $g\in C(X, \mathbb{R})$ such that $0\leq g\leq 1$, $g(x)=0$ for $x\in K$, and  $g(x)=1$ for $x\in E$. This shows $p(g)=0$, and  hence $\mu(E)\leq\Psi(g)\leq p(g)=0$. We get $\mu(K)=1$ by the regularity of $\mu$.

Finally, we show that $$\mu(B_{m}^{\alpha}(x,\varepsilon))\leq \frac{1}{c}\exp\left(-s \cdot S_mf(x)\right)$$ for all $x\in X$ and $m\geq n$. Again,  by Urysohn's lemma for any compact set $E\subset B_{m}^{\alpha}(x,\varepsilon)$, there exists $g\in C(X, \mathbb{R})$ such that  $0\leq g\leq 1$, $g(y)=0$ for $y\in X\backslash B_{m}^{\alpha}(x,\varepsilon)$ and $g(y)=1$ for $y\in E$. Thus, $\mu(E)\leq\Psi(g)\leq p(g).$
 Notice  that $\chi_{K}\cdot g\leq\chi_{B_{m}^{\alpha}(x,\varepsilon)}$ and $m\geq n$. Then $$W^{\alpha}(\chi_{K}\cdot g,s,\varepsilon,n,\varphi)\leq \exp\left(-s \cdot S_mf(x)\right),$$
  which implies that $\mu(E)\leq\ p(g)\leq\frac{1}{c}\exp\left(-s \cdot S_mf(x)\right)$. Then it follows  from  the regularity of $\mu$ that
$$\mu(B_{m}^{\alpha}(x,\varepsilon)) \leq \frac{1}{c}\exp\left(-s \cdot S_mf(x)\right)$$
for all $x\in X$ and $m\geq n$.
\end{proof}

\subsubsection{Variational principle}

With the help  of  Billingsley Type Theorem of  $\alpha$-BS dimension and Frostman's lemma of  weighted $\alpha$-BS  dimension, we give the proof of Theorem \ref{thm 1.2}.

\begin{proof}[Proof of Theorem \ref{thm 1.2}]
Let $0<s<{\rm dim}_{BS}^{\alpha}(f,K,\varphi)$. Then, by  Theorem \ref{t34} there exist $\varepsilon>0$ and $n\in\mathbb N$ such that
$c:=W^{\alpha}(K,s,\varepsilon,n,\varphi)>0.$
Using Lemma \ref{l36}, there exists $\mu\in\mathcal{M}(X)$ with $\mu(K)=1$ such that
$$\mu(B_{m}^{\alpha}(x,\varepsilon))\leq\frac{1}{c}\exp\left(-s \cdot S_mf(x)\right)~\text{for any}~ x\in X, m\geq n.$$
Thus, for any $x\in X$, we have
$${P}_{\mu}^{\alpha}(x,f,\varphi)\geq\liminf_{m\rightarrow\infty}\frac{-\log\mu({B_{m}^{\alpha}(x,\varepsilon))}}
{S_mf(x)}\geq s.$$
Therefore, ${P}_{\mu}^{\alpha}(f,\varphi)=\int_{X} {P}_{\mu}^{\alpha}(x,f,\varphi)d\mu(x)\geq s$. This proves that
$${\rm dim}_{BS}^{\alpha}(f,K,\varphi)\leq\sup\{{P}_{\mu}^{\alpha}(f,\varphi):\mu\in\mathcal{M}(X),\mu(K)=1\}.$$

Now we prove the reverse inequality:
 $${\rm dim}_{BS}^{\alpha}(f,K,\varphi)\geq\sup\{{P}_{\mu}^{\alpha}(f,\varphi):\mu\in\mathcal{M}(X),\mu(K)=1\}.$$
Fix $\mu\in\mathcal{M}(X)$ with $\mu(K)=1$. Let $0<s<{P}_{\mu}^{\alpha}(f,\varphi)=\int_{K} {P}_{\mu}^{\alpha}(x,f,\varphi)d\mu(x)$. Then there is a Borel  set $K_0 \subset K$  with  positive $\mu$-measure such that
$${P}_{\mu}^{\alpha}(x,f,\varphi) >s$$
for all $x\in K_0$. By the  Billingsley Type Theorem stated in  Theorem \ref{t33}, we have ${\rm dim}_{BS}^{\alpha}(f,K,\varphi)\geq {\rm dim}_{BS}^{\alpha}(f,K_0,\varphi) \geq s$. Letting  $s \to {P}_{\mu}^{\alpha}(f,\varphi)$, we get
$${\rm dim}_{BS}^{\alpha}(f,K,\varphi)\geq {P}_{\mu}^{\alpha}(f,\varphi)$$
for all  $\mu\in\mathcal{M}(X)$ with $\mu(K)=1$. This completes the proof.
\end{proof}

\subsection{Linking $\alpha$-Bowen topological entropy and $\alpha$-topological entropy}

In this subsection, we prove a variational  inequality for $\alpha$-topological entropy.

Let $(X,d,f)$ be a  TDS. A set $E\subset X$ is  called an \emph{$(n, \alpha,\varepsilon)$-spanning set of $X$} if  any $x\in X$, there exists  $y\in E$ such that $d_n^{\alpha}(x,y)<\varepsilon$. Let  $r_n(f,\alpha,X,\varepsilon)$  denote the smallest cardinality of  $(n, \alpha,\varepsilon)$-spanning sets of $X$.   Put
$$h_{top}^{\alpha}(f, X,\varepsilon)=\limsup_{n\to \infty} \frac{1}{n} \log r_n(f,\alpha,X,\varepsilon).$$

\begin{defn}\cite{k13}
We  define the $\alpha$-topological   entropy of $X$  as
$$h_{top}^{\alpha}(f, X)=\lim_{\varepsilon
\to 0}h_{top}^{\alpha}(f, X,\varepsilon).$$
\end{defn}

\begin{rem}
If $\alpha =0$, the above definition  recovers the classical topological entropy $h_{top}(f,X)$ of $X$.
\end{rem}

\begin{pro}\label{thm 1.3}
Let $(X,d,f)$ be a TDS. Then for any $\alpha \geq 0$,
$$h_{top}^{\alpha,B}(f,X)\leq h_{top}^{\alpha}(f, X),$$
where $h_{top}^{\alpha,B}(f,X)$ denotes the $\alpha$-Bowen topological entropy of $X$.
\end{pro}

\begin{proof}
For any $s > h_{top}^{\alpha}(f, X) $, there exists $\varepsilon_0>0$  such that  for every $0<\varepsilon <\varepsilon_0$,
$$r_n(f,\alpha,X,\varepsilon)<e^{ns}$$
for all sufficiently large $n$. Fix such  $\varepsilon$ and $n$. If $E$  is an $(n,\alpha,\varepsilon)$-spanning set of $X$,  then  $X=\cup_{x\in E} B_n^{\alpha}(x,\varepsilon)$. Hence,
$$M^{\alpha}(X,s,\varepsilon,n,1)\leq e^{ns} \cdot e^{-ns}=1.$$
This implies that $h_{top}^{\alpha,B}(f,X) \leq s$.  Letting $s \to h_{top}^{\alpha}(f, X) $, we get $h_{top}^{\alpha,B}(f,X)\leq   h_{top}^{\alpha}(f, X)$.

\end{proof}

As the corollary of Theorems \ref{thm 1.2} and   Proposition \ref{thm 1.3},  for any TDS  we  have a variational  inequality for $\alpha$-topological entropy.

\begin{thm}\label{thm 4.4}
Let $(X,d,f)$  be a TDS. Then for all $\alpha  > 0$,
$$\sup_{\mu  \in M(X,f)}\underline{h}_{\mu}^{BK}(f,\alpha)\leq  h_{top}^{\alpha}(f, X).$$
where $$\underline{h}_{\mu}^{BK}(f,\alpha)=\int_{X}\lim_{\varepsilon\rightarrow0}\liminf_{n\rightarrow\infty} \frac{-\log\mu({B_{n}^{\alpha}(x,\varepsilon))}}
{n}d\mu(x)$$ 
denotes the lower $\alpha$-local   Brin-Katok entropy of $\mu$.
\end{thm}

\begin{rem}
$(1)$ 
Similar to the proof of the lower bound of the classical variational principle for topological entropy (cf. \cite[Theorem 8.6, p.189]{w82}),  Theorem \ref{thm 4.4} was first proved by  Kawan \cite[Theorem 4]{k13} for a  $\mathcal{C}^{1+\varepsilon}$-diffeomorphism $f$ on a
 smooth compact manifold $X$. Here, without replying on the differential structures we give a different approach to extend this inequality to all TDSs.

(2) For expansive  dynamical systems, the authors in \cite[Corollary 1.6]{ctx25} proved   that   there exists a metric on phase space such that if $\alpha$  varies in a bounded interval of $\mathbb{R}$, then the equality  in  Theorem  \ref{thm 4.4} holds. It is still an open question  whether the equality is  valid for all TDSs. 
\end{rem}

\section{$\alpha$-topological entropy and Hausdorff dimension of symbolic systems} \label{sec 4}

In this section,  we  give  examples of symbolic systems to illustrate  how  $\alpha$-BS dimension varies  in  the parameter $\alpha$, and prove Theorem \ref{thm 1.4}.

We first recall the essential  setting and  relevant  terminologies of symbolic systems.  Let $\{0,...,k-1\}$ be  $k$-symbols endowed with the discrete topology. Let $\Sigma_k^{+}=\{0,...,k-1\}^{\mathbb{N}}$ be the product space, whose product topology is metrizable by the metric $$d(x,y)=\mathrm{e}^{-n(x,y)},$$
where $n(x,y)$ is  an integer that $x,y$ first disagree; we let $n(x,y)=\infty$ if $x=y$.

Let $\sigma:\Sigma_k^{+}\rightarrow \Sigma_k^{+}$ be the left shift given by $\sigma((x_n)_n)=(x_{n+1})_n$.  Then $(\Sigma_k^{+},d,\sigma)$ is called the \emph{full 1-sided} shift over $\{0,...,k-1\}$. Let $A$ be an  incidence matrix of size of $k\times k$ with entries  either $0$ or $1$. We always assume that  every row of $A$ at lest has  one $1$. The \emph{subshift of finite type} is the  $\sigma$-invariant  closed  subset $$\Sigma_A^{+}:=\{x\in \Sigma_k^{+}: A_{x_n,x_{n+1}}=1~\forall~n \in \mathbb{N}\}$$
with the left shift map $\sigma$ restricted to $\Sigma_A^{+}$. For $x=(x_0,x_1,...) \in \Sigma_A^{+}$,  let $$[x_0,..,x_n]:=\{y\in \Sigma_A^{+}: y_j=x_j ~\text{for}~ j=0,..,n\}$$    be a cylinder set of $x$ with the length $n+1$. By $(x_0,x_1,...,x_n)$  we mean an $A$-admissible $n$-word if $A_{x_j,x_{j+1}}=1$ for all $0\leq j\leq n-1$.  Denote by $\Sigma_A^n$ the set of all $A$-admissible $n$-words.

Now  we are in a position to  prove Theorem \ref{thm 1.4}.

\begin{proof}[Proof of Theorem \ref{thm 1.4}]

(1). For any $0<\varepsilon<1$, there exists a positive integer $N_0$ such that
$$(\frac{1}{\mathrm{e}} )^{N_0+1}<\varepsilon\leq (\frac{1}{\mathrm{e}} )^{N_0}.$$
Then for  all sufficiently large $n$,
\begin{align}\label{equ 4.1}
	[x_0,...,x_{n+\lfloor n\alpha \rfloor+N_0+2}]\subset B_n^{\alpha}(x,\varepsilon) \subset [x_0,...,x_{n+\lfloor \alpha(n-1) \rfloor+N_0+1}]
\end{align}
holds for every $x\in \Sigma_A^{+}$, where  $\lfloor \beta \rfloor:=\max\{n\in \mathbb{Z}: n\leq \beta\}$.

  Recall that in symbolic systems, it is well-known that the Hausdorff dimension can be equivalently defined by the cylinder sets of symbolic systems. Namely,  for every $E \subset \Sigma_A^{+}$, the Hausdorff dimension of $E$ is given by 
\begin{align*}
	{\rm dim}_H(E,d)=\inf\{s: \widetilde{\mathcal{H}}^s(E)=0\}= \sup\{s: \widetilde{\mathcal{H}}^s(E)=\infty\},
\end{align*}
where 
$\widetilde{\mathcal{H}}^s(E)=\lim_{N\to \infty}\widetilde{\mathcal{H}}_N^s(E)$ and
$$\widetilde{\mathcal{H}}_N^s(E):=\inf\{\sum_{i=1}^{\infty} \mathrm{e}^{-n_is}\},$$
where the infimum ranges over all   countable families of cylinder sets $E_i=[x_0,x_1,...,x_{n_i}]$ of $\Sigma_A^{+}$  such that  $E\subset \cup_{i=1}^{\infty}E_i$ and $n_i \geq N$.

Now we show  that
$$h_{top}^{\alpha,B}(\sigma,E)=(1+\alpha){\rm dim}_H(E,d).$$

Fix $0<\varepsilon<1$ and take $N_0$ as in (\ref{equ 4.1}). Fix $n\geq N_0$.  Choose  a  countable family  of $\alpha$-Bowen balls  $\{B_{n_i}^{\alpha}(x_i,\varepsilon)\}_{i\in I}$  covering $E$ with $x_i\in \Sigma_A^{+}$ and $n_i\geq n$.  Then, by  (\ref{equ 4.1}), $E$ can be covered by  the cylinder sets of $\Sigma_A^{+}$, i.e.,
$$E\subset \cup_{i\in I} [x_0^i,...,x_{n_i+\lfloor \alpha(n_i-1) \rfloor+N_0+1}^i]$$ for $x_i=(x_0^i,x_1^i,...) \in \Sigma_A^{+}.$ Then for all $s> 0$,
\begin{align*}
	\widetilde{\mathcal{H}}_{n+\lfloor \alpha(n-1) \rfloor+N_0+1}^s(E) \leq & \sum_{i\in I}\mathrm{e}^{-(n_i+\lfloor \alpha(n_i-1) \rfloor+N_0+1)s}\\
	 < & \sum_{i\in I}\mathrm{e}^{-(n_i+\lfloor \alpha(n_i-1) \rfloor)s}\\
	<&\sum_{i\in I}\mathrm{e}^{-((1+\alpha)(n_i-1))s}\\
	=&\mathrm{e}^{(1+\alpha)s}\sum_{i\in I}\mathrm{e}^{-((1+\alpha)s)n_i}.
\end{align*}
Hence,  $$\widetilde{\mathcal{H}}_{n+\lfloor \alpha(n-1) \rfloor+N_0+1}^s(E)\leq \mathrm{e}^{(1+\alpha)s} \cdot M^{\alpha}(E,(1+\alpha)s,\varepsilon,n,1).$$ Letting $n \to \infty$, we obtain that
$$\widetilde{\mathcal{H}}^s(E)\leq  \mathrm{e}^{(1+\alpha)s} \cdot M^{\alpha}(E,(1+\alpha)s,\varepsilon,1).$$ If $\beta > M^{\alpha}(E,\varepsilon,1)$, we have $\widetilde{\mathcal{H}}^{\frac{\beta}{1+\alpha}}(E)=0$, and hence ${\rm dim}_H(E,d)\leq  \frac{\beta}{1+\alpha}$. Then we deduce that $$(1+\alpha){\rm dim}_H(E,d)\leq  h_{top}^{\alpha,B}(\sigma,E).$$

Again, fix $0<\varepsilon<1$ and take $N_0$ as in (\ref{equ 4.1}).
Let $N\geq N_0$. Take  a   countable family of cylinder sets $E_i=[x_0,x_1,...,x_{n_i}]$ of $\Sigma_A^{+}$  such that  $E\subset \cup_{i=1}^{\infty}E_i$ and $n_i \geq \lfloor N(1+\alpha) \rfloor+N_0+3$.   For any $n_i\geq \lfloor N(1+\alpha) \rfloor+N_0+3$,  there is a unique $N_i\geq N$ such that
$$\lfloor N_i(1+\alpha) \rfloor+N_0+3<n_i\leq \lfloor (N_i+1)(1+\alpha)  \rfloor+N_0+3.$$
Notice that $\lfloor (N_i+1)(1+ \alpha)  \rfloor-\lfloor N_i(1+ \alpha)  \rfloor \leq \alpha +2$. We  rewrite $n_i$ as $n_i=\lfloor N_i(1+\alpha) \rfloor+q_i+N_0+3$, where $0\leq q_i\leq \lfloor \alpha+2  \rfloor$. Since $n+\lfloor n\alpha \rfloor \leq \lfloor N(1+\alpha) \rfloor+1$, by (\ref{equ 4.1})  for sufficiently large $n$, one has
$$[x_0,...,x_{\lfloor N(1+\alpha) \rfloor+N_0+3}]\subset B_n^{\alpha}(x,\varepsilon)$$
for all $x\in \Sigma_A^{+}$.
Then  $E$
can be  also covered by $\alpha$-Bowen balls:
$$E\subset \cup_{i=1}^{\infty}[x_0,...,x_{n_i-q_i}] \subset \cup_{i=1}^{\infty}B_{N_i}^{\alpha}(y_i,\varepsilon),$$
for some $y_i\in [x_0,...,x_{n_i-q_i}]$. Therefore, for all $s> 0$,
\begin{align*}
	\sum_{i=1}^{\infty}\mathrm{e}^{-n_is}&\geq  \mathrm{e}^{-(\lfloor \alpha +2 \rfloor+N_0+3)}\sum_{i=1}^{\infty}e^{- s \lfloor  N_i(1+\alpha) \rfloor}\\
	&\geq  \mathrm{e}^{-(\lfloor \alpha +2 \rfloor+N_0+3)}\sum_{i=1}^{\infty}e^{-N_i(1+\alpha)s}\\
	&\geq  \mathrm{e}^{-(\lfloor \alpha +2 \rfloor+N_0+3)} \cdot M^{\alpha}(E,(1+\alpha)s,\varepsilon,N,1).
\end{align*}
This yields that  $$ \mathrm{e}^{(\lfloor \alpha +2 \rfloor+N_0+3)}\cdot \widetilde{\mathcal{H}}_{\lfloor N(1+\alpha) \rfloor+N_0+3}^s(E)\geq M^{\alpha}(E,(1+\alpha)s,\varepsilon,N,1).$$ Letting $N \to \infty$, we get
$ \mathrm{e}^{(\lfloor \alpha +2 \rfloor+N_0+3)}\cdot \widetilde{\mathcal{H}}^s(E)\geq M^{\alpha}(E,(1+\alpha)s,\varepsilon,1).$ This allows us to get
$$(1+\alpha){\rm dim}_H(E,d)\geq  h_{top}^{\alpha,B}(\sigma,E).$$

(2).  For every word $(x_0,...,x_{n+\lfloor n\alpha \rfloor+N_0+2}) \in \Sigma_A^{n+\lfloor n\alpha \rfloor+N_0+2}$, choose a point  $\tilde{x} \in \Sigma_A^{+}$ such that $\tilde{x}_j=x_j$ for all $0\leq j \leq  x_{n+\lfloor n\alpha \rfloor+N_0+2}$. Then, by (\ref{equ 4.1}) we have
 \begin{align}\label{inequ 4.2}
r_n(\sigma,\alpha,\Sigma_A^{+},\varepsilon)\leq \#\Sigma_A^{n+\lfloor n\alpha \rfloor+N_0+2}.
\end{align}
We define the matrix  norm of $A$ to be  $||A||=\sum_{i,j=0}^{k-1}|A_{i,j}|$. It is well-known that the spectral radius of $A$ is given by $r(A)=\lim_{n \to \infty}||A^n||^{\frac{1}{n}}$. Notice that  $||A^n||=\#\Sigma_A^{n}$ for every $n\geq 1$ and  the sequence $\{ \log \#\Sigma_A^{n}\}_n$ is sub-additive in $n$. Hence,
\begin{align*}
\log r(A)=\lim_{n\to \infty} \frac{\log ||A^n||}{n}=\lim_{n\to \infty} \frac{\log \#\Sigma_A^{n}}{n}.
\end{align*}
Now, by (\ref{inequ 4.2}) we deduce that
$$h_{top}^{\alpha}(\sigma, \Sigma_A^{+},\varepsilon) \leq \limsup_{n \to \infty} \frac{ \log \#\Sigma_A^{n+\lfloor n\alpha \rfloor+N_0+2} }{n}=(1+\alpha)\log r(A).$$
Letting $\varepsilon \to 0$,  we have
\begin{align}\label{inequ 4.3}
h_{top}^{\alpha}(\sigma, \Sigma_A^{+}) \leq  (1+\alpha) \log r(A).
\end{align}
Similarly, we obtain that $ h_{top}^{\alpha}(\sigma, \Sigma_A^{+}) \geq (1+\alpha) \log r(A).$ Hence, we obtain that 
$$h_{top}^{\alpha}(\sigma, \Sigma_A^{+}) = (1+\alpha) \log r(A).$$

In particular, if $\alpha =0$, we have $  h_{top}^{B}(\sigma,\Sigma_A^{+})={\rm dim}_H(\Sigma_A^{+},d)$ and $h_{top}(\sigma, \Sigma_A^{+}) =  \log r(A)$.  By  \cite[Proposition 1.1]{b73}, it states that $ h_{top}^{B}(\sigma,E)=h_{top}(\sigma, \Sigma_A^{+})$. This shows that  ${\rm dim}_H(\Sigma_A^{+},d)=\log r(A)$ and hence $h_{top}^{\alpha,B}(\sigma,\Sigma_A^{+})=h_{top}^{\alpha}(\sigma, \Sigma_A^{+})$ for every $\alpha \geq 0$.
\end{proof}

\begin{rem}
The dependence on parameters $\alpha$ of the Hausdorff dimension of fractal invariant sets was studied with the help of pressure, for expanding maps by \cite{RB, rue82}, for hyperbolic diffeomorphisms by \cite{mm83}, and for hyperbolic non-invertible maps by \cite{ms13,mu10}.
\end{rem}

Using  Theorem  \ref{thm 1.4}, we extend  several classical  results in dimension theory \cite{b73,pp84,tv03,cts05,ps07} to $\alpha$-topological entropy of full shifts over finite symbols.

\begin{cor}\label{cor 4.1}
Let $(\Sigma_k^{+},d,\sigma)$ be the full 1-sided shift. Then for every $\alpha \geq 0$,

$(1)$  $h_{top}^{\alpha,B}(\sigma,\Sigma_k^{+})=h_{top}^{\alpha}(\sigma, \Sigma_k^{+}) = (1+\alpha) \log k.$

$(2)$   ${\rm dim}_H(\Sigma_k^{+},d)=\log k$.

$(3)$ The $\alpha$-topological entropy satisfies the following variational principle: $$h_{top}^{\alpha}(\sigma, \Sigma_k^{+})=\max_{\mu \in M(\Sigma_k^{+},\sigma)}\{(1+\alpha)h_{\mu}(\sigma)\}.$$

$(4)$ For any $\mu \in M(\Sigma_k^{+}, \sigma)$,
$\underline{h}_{\mu}^{BK}(\sigma,\alpha)=(1+\alpha)h_{\mu}(\sigma),$
where
$$ \underline{h}_{\mu}^{BK}(\sigma,\alpha)=\int_{\Sigma_k^{+}}\lim_{\varepsilon\rightarrow0}\liminf_{n\rightarrow\infty}
-\frac{\log\mu(B_{n}^{\alpha}(x,\varepsilon))}{n}d\mu(x)$$ is the lower $\alpha$-local Brin-Katok entropy of $\mu$.

$(5)$  For any $\mu\in M(\Sigma_k^{+}, \sigma)$, $h_{top}^{\alpha,B}(\sigma,G_{\mu})=(1+\alpha)h_{\mu}(\sigma),$
where
$$G_{\mu}:=\{x\in \Sigma_k^{+}:\lim_{n \to \infty}\frac{1}{n}\sum_{j=0}^{n-1}g(x_j)=\int g d\mu ~\text{for  all}~g\in C(\Sigma_k^{+},\mathbb{R})\}$$ is the set of  generic points of $\mu$.

$(6)$   Let $g\in C(\Sigma_k^{+},\mathbb{R})$. For $s\in \mathbb{R}$, define the $s$-level set of $g$ and the set of divergent points of $g$ as
\begin{align*}
K_s(g):=\{(x_n)\in \Sigma_k^{+}:\lim_{n \to \infty}\frac{1}{n}\sum_{j=0}^{n-1}g(x_j)=s\}
\end{align*}
and $I_{g}:=\{(x_n)\in \Sigma_k^{+}: ~\text{the limit}~\frac{1}{n}\sum_{j=0}^{n-1}g(x_j)~\text{does not exist}\}$ respectively.

If $K_s(g)\not=\emptyset$ and $I_g\not=\emptyset$, then
\begin{align*}
h_{top}^{\alpha,B}(\sigma,K_s(g))&=\sup\{(1+\alpha)h_{\mu}(\sigma):\mu \in M(\Sigma_k^{+}, \sigma),\int gd\mu=s\},\\
h_{top}^{\alpha,B}(\sigma,I(g))&=h_{top}^{\alpha}(\sigma, \Sigma_k^{+}).
\end{align*}
\end{cor}

\begin{proof}We prove these statements one by one.

(1) and (2) are the direct consequence of Theorem \ref{thm 1.4}.

(3).  Consider the product   measure $\mu$  on $\Sigma_k^{+}$ generated by $(\frac{1}{k},\frac{1}{k},...,\frac{1}{k})$-probability vector that equally assigns $\frac{1}{k}$ to each symbol.  By (\ref{equ 4.1}),  for every $x\in \Sigma_k^{+}$ one has
$$\mu(B_n^{\alpha}(x,\varepsilon)) \leq \mu([x_0,...,x_{n+\lfloor \alpha(n-1) \rfloor+N_0+1}])\leq  (\frac{1}{k})^{n+\lfloor \alpha(n-1) \rfloor+N_0+2}.$$
Then for any $x\in \Sigma_k^{+}$,
${P}_{\mu}^{\alpha}(\sigma,x,1)\geq (1+\alpha)\cdot  \log k,$
and hence  ${P}_{\mu}^{\alpha}(\sigma,1)\geq (1+\alpha)\cdot  \log k$. By   Theorem \ref{thm 1.2}, we have  $$h_{top}^{\alpha,B}(\sigma,\Sigma_k^{+})\geq {P}_{\mu}^{\alpha}(\sigma,1)\geq (1+\alpha)\cdot  \log k =h_{top}^{\alpha}(\sigma, \Sigma_k^{+}),$$
where we used (1) for the last equality.

(4). Applying  the Shannon-McMillan-Breiman Theorem to the natural generator $\mathcal{P}=\{[0],[1],...,[k-1]\}$,   by (\ref{equ 4.1}),  for all $\mu\in M(\Sigma_k^{+},\sigma)$ we have
$$\underline{h}_{\mu}^{BK}(\sigma,\alpha)=(1+\alpha)h_{\mu}(\sigma,\mathcal{P}) =(1+\alpha)h_{\mu}(\sigma).$$

(5). It is well-known  that the symbolic systems have specification property, and hence admit $g$-almost property \cite[Proposition 2.1]{ps07}. For any  TDS $(X,d,f)$,  Pfister and  Sullivan  \cite[Theorem 1.3]{ps07} proved that  for any $\mu \in M(X,f)$, one has $h_{top}^B(f,G_{\mu})=h_{\mu}(f).$ By Theorem \ref{thm 1.4} (1), we have
$$h_{top}^{\alpha,B}(\sigma,G_{\mu})=(1+\alpha)h_{top}^{B}(\sigma,G_{\mu})=(1+\alpha)h_{\mu}(\sigma).$$

(6). In \cite[Theorem 3.5]{tv03} and \cite[Theorem 3.1]{cts05}, for systems with the specification property the authors respectively proved that   if $K_s(g)$ and $I(g)$ are not empty, then
\begin{align*}
h_{top}^{B}(f,K_s(g))&=\sup\{h_{\mu}(f):\mu \in M(X, f),\int gd\mu=s\},\\
h_{top}^{B}(f,I(g))&=h_{top}(f, X).
\end{align*}
Combining this fact with Theorem \ref{thm 1.4} (1), we  get the desired  equalities.
\end{proof}

\begin{rem}
Under  Corollary \ref{cor 4.1}, (3) is the content of classical variational principle for $\alpha$-topological entropy. Around the Birkhoff's ergodic theorem, (6) gives a conditional variational principle for $\alpha$-Bowen topological entropy of level sets, and  verifies that although the set of divergent points is negligible in the  measure-theoretic  sense, its topological complexity  can be sufficiently large, even  equals the $\alpha$-topological entropy of the whole phase space.
\end{rem}

\section{Open questions}\label{sec 5}

Based on our  main results,  in  this section we pose some open questions  to help us understand  the interplay between the topological dynamical systems and  the dimension theory of $\alpha$-entropy-like quantities.

\textbf{Question:}
Let  $(X,d,f)$ be a TDS.  Does there  exist the non-negative functions $k_1(\alpha)$, $k_2(\alpha)$, $k_3(\alpha)$ on $\mathbb{R}_{\geq 0}$ such that for  any $\alpha \geq 0$,
\begin{align*}
h_{top}^{\alpha,B}(f,X)&=K_1(\alpha) \cdot h_{top}^{B}(f,X),\\
 h_{top}^{\alpha}(f, X)&=K_2(\alpha) \cdot h_{top}(f,X),\\
  \underline{h}_{\mu}^{BK}(f,\alpha)&=K_3(\alpha) \cdot h_{\mu}(f)~\text{for all}~\mu \in M(X,f);
\end{align*}
in particular, whether $k_j(\alpha) \to 1$ as $\alpha \to 0$ for $j=1,2,3$.



 \section*{Appendix: comparing $\alpha$-topological entropy with neutralized topological entropy}\label{sec 6}
 In this section, we  compare  the $\alpha$-topological entropy with the neutralized topological entropy introduced in \cite{ycz24}, and show this two kind of entropy do not coincide in symbolic systems.

 Using  the $\varepsilon$-neutralized  Bowen open ball\footnote[3]{It manifests as the estimation of the asymptotic behaviors of measures with a distinctive geometric shape, when Ovadia and Rodriguez-Hertz  \cite{orh24} define Brin-Katok entropy by using $\varepsilon$-neutralized Bowen open balls to neutralize sub-exponential sets with such a shape. It finds application in describing the neighborhood with a local linearization of the dynamics \cite{orh24, dq25}. See also \cite{ycz24,ctx25} for further results about   the  relation of neutralized topological entropy and dimension theory.} \cite{orh24}:
  $$B_n(x,\mathrm{e}^{-n\varepsilon}):=\{y\in X:d(f^jx,f^jy)<\mathrm{e}^{-n\varepsilon},~ \text{for}~{j=0,1,...,n-1}\},$$
  Yang, Chen and Zhou \cite{ycz24}  defined  the neutralized $\varepsilon$-topological entropy of $X$
 as
  $$\widetilde{h}_{top}(f,X,\varepsilon)=\limsup_{n \to \infty}\frac{\log r_n(X,\mathrm{e}^{-n\varepsilon})}{n},$$
  where  $r_n(X,\mathrm{e}^{-n\varepsilon})$ is the smallest cardinality of  the $\varepsilon$-neutralized  Bowen open balls $B_n(x,\mathrm{e}^{-n\varepsilon})$ needed to cover $X$.

  The \emph{neutralized topological entropy of $X$}  is defined by
  $$\widetilde{h}_{top}(f,X)=\lim_{\varepsilon \to 0}\widetilde{h}_{top}(f,X,\varepsilon)=\inf_{\varepsilon >0}\widetilde{h}_{top}(f,X,\varepsilon).$$

Compared  with the  classical topological entropy, a remarkable feature for the  neutralized topological entropy and the $\alpha$-topological entropy is  that they are  all  defined using Bowen balls with  varied radii. We show that   their   topological complexity  is different  even if we consider  it in  the symbolic systems.


 \begin{pro}
 In the  context of Corollary \ref{cor 4.1},  for every $\alpha > 0$, we have
 $$\widetilde{h}_{top}(\sigma,\Sigma_k^{+})=\log k <(1+\alpha)\log k =  h_{top}^{\alpha}(\sigma,\Sigma_k^{+}).$$
 \end{pro}

\begin{proof}
It is easy to see that
$$ \widetilde{h}_{top}(\sigma,\Sigma_k^{+})=\lim_{l \to \infty}\limsup_{m \to \infty}\frac{\log r_{ml}(\Sigma_k^{+},\mathrm{e}^{-ml \cdot \frac{1}{l}})}{ml}=\lim_{l \to \infty}\limsup_{m \to \infty}\frac{\log r_{ml}(\Sigma_k^{+},\mathrm{e}^{-m})}{ml}.$$

Fix a positive integer $l$. According to  the definition  of $d$,  for every $m\in \mathbb{N}$ one  has  $r_{ml}(\Sigma_k^{+},e^{-m})\leq k^{ml+m+1}$. This implies that $\widetilde{h}_{top}(\sigma,\Sigma_k^{+},\frac{1}{l})\leq (1+\frac{1}{l})\log k.$
Hence, $$\widetilde{h}_{top}(\sigma,\Sigma_k^{+}) \leq \log k$$ by letting $l\to \infty$. Besides,  by  the fact that for each fixed $\varepsilon>0$, $B_n(x,\mathrm{e}^{-n\varepsilon})\subset B_n(x,\varepsilon)$ for sufficiently large $n$,  we have $\widetilde{h}_{top}(\sigma,\Sigma_k^{+}) \geq \log k$.    This shows that  $\widetilde{h}_{top}(\sigma,\Sigma_k^{+}) =\log k< h_{top}^{\alpha}(\sigma,\Sigma_k^{+})$.,
\end{proof}
Recall that $h_{top}(f,X)=\lim_{\varepsilon \to 0}  h_{top}^{0}(f, X,\varepsilon)$. For  the certain conditions, $h_{top}(f,X)$ can be expressed as the limit of  $\alpha$-topological $\varepsilon$-entropy by letting $\alpha \to 0$ and $\varepsilon \to 0$.

\begin{thm}
Let $(X,d,f)$ be a TDS.  If $\widetilde{h}_{top}(f,X)= h_{top}(f,X)$, then
$${h}_{top}(f,X)=\lim_{\varepsilon \to 0}\lim_{\alpha \to 0}h_{top}^{\alpha}(f, X,\varepsilon)=\sup_{\varepsilon >0}\inf_{\alpha >0}h_{top}^{\alpha}(f, X,\varepsilon).$$
\end{thm}

\begin{proof}
Fix $\varepsilon >0$. Choose a sufficiently large positive integer $n$  such that $\mathrm{e}^{-\frac{\varepsilon}{2}n}<\frac{\varepsilon}{2}$. If $x,y\in X$ with $d_n(x,y)<\mathrm{e}^{-n\varepsilon}$, then
$$d_n(x,y)<\frac{\varepsilon}{2}\mathrm{e}^{-\frac{\varepsilon}{2}n}<\frac{\varepsilon}{2}\mathrm{e}^{-\alpha n}$$
for any $0<\alpha <\frac{\varepsilon}{2}$. Hence, for any $0\leq j<n$, $\mathrm{e}^{\alpha j}d(f^jx,f^jy)<\frac{\varepsilon}{2}$. This shows  that $d_n^{\alpha}(x,y)<\frac{\varepsilon}{2}$. Then $ r_n(f,\alpha,X,\frac{\varepsilon}{2}) \leq r_n(X,\mathrm{e}^{-n\varepsilon})$. This means that
$$h_{top}^{\alpha}(f, X,\frac{\varepsilon}{2})\leq  \widetilde{h}_{top}(f,X,\varepsilon).$$
Letting $\alpha \to 0$ and then letting $\varepsilon \to 0$, we get
$$\lim_{\varepsilon \to 0}\lim_{\alpha \to 0}h_{top}^{\alpha}(f, X,\varepsilon)\leq \widetilde{h}_{top}(f,X).$$

Notice that  if $d_n^{\alpha}(x,y)<\varepsilon$, then $d_n(x,y)<\varepsilon$. This implies that  for any $\alpha >0$ and $\varepsilon >0$, $$h_{top}^{0}(f, X,\varepsilon)\leq h_{top}^{\alpha}(f, X,\varepsilon).$$

Then, together with the fact  $\widetilde{h}_{top}(f,X)= h_{top}(f,X)$, we can deduce that
$$h_{top}(f,X)\leq  \lim_{\varepsilon \to 0}\lim_{\alpha \to 0}h_{top}^{\alpha}(f, X,\varepsilon)\leq \widetilde{h}_{top}(f,X)=h_{top}(f,X).$$
\end{proof}

\subsection*{Acknowledgment}
We thank the anonymous referee for many valuable comments and suggestions that greatly improved the  quality of the paper.

\end{document}